\documentclass[11pt]{amsart}
\usepackage{verbatim}
\usepackage{amsmath}
\usepackage{enumerate}
\usepackage{amssymb}
\usepackage{color}
\usepackage{url}
\usepackage{graphicx}
\usepackage{fullpage}
\graphicspath{ {images/} }

\newcommand\+{\;\lower\plusheight\hbox{$+$}\;}
\newcommand\lldots{\;\lower\plusheight\hbox{$\cdots$}\;}

\newtheorem{Theorem}{Theorem}[section]
\newtheorem{Lemma}[Theorem]{Lemma}
\newtheorem{Corollary}[Theorem]{Corollary}
\newtheorem{Definition}[Theorem]{Definition}
\newtheorem{Example}[Theorem]{Example}
\newtheorem{Remark}[Theorem]{Remark}

\def\Re{\mathop{\rm Re}\nolimits}

\def\Im{\mathop{\rm Im}\nolimits}

\setbox0=\hbox{$+$}
\newdimen\plusheight
\plusheight=\ht 0 \setbox0=\hbox{$-$}
\newdimen\minusheight
\minusheight=\ht0

\setbox0=\hbox{$\cdots$}
\newdimen\cdotsheight
\cdotsheight=\plusheight

\font\cute=cmitt10 at 12pt 
\newcommand{\kay}{{\text{\cute k}}}

\newcommand{\Z}{\mathbb Z}
\newcommand{\Q}{\mathbb Q}

\newcommand{\SO}{\operatorname{SO}}
\newcommand{\SL}{\operatorname{SL}}
\newcommand{\Mp}{\operatorname{Mp}}

\begin{document}
\title{Borcherds Products on Unitary Group $U(2,1)$}
\author{Tonghai Yang and Dongxi Ye}

\thanks{The research of the first author is supported by a NSF grant DMS-1500743.}
\address{
Department of Mathematics, University of Wisconsin\\
480 Lincoln Drive, Madison, Wisconsin, 53706 USA}
\email{thyang@math.wisc.edu}
\address{
Department of Mathematics, University of Wisconsin\\
480 Lincoln Drive, Madison, Wisconsin, 53706 USA}
\email{lawrencefrommath@gmail.com}

\begin{abstract}
In this note, we construct canonical bases for the spaces of weakly holomorphic
modular forms with poles supported at the cusp $\infty$ for $\Gamma_{0}(4)$ of
integral weight $k$ for $k\leq-1$, and we make use of the basis elements for the case
$k=-1$ to construct explicit Borcherds products on unitary group $U(2,1)$.
\end{abstract}

\maketitle
\allowdisplaybreaks
\numberwithin{equation}{section}
\section{Introduction}

In 1998, Borcherds developed a new method to produce memomorphic modular forms on an orthogonal Shimura variety from weakly holomorphic classical modular forms via regularized theta liftings. These memomorphic modular forms have two distinct properties. The first one is the so-called Boorcherds product expansion at a cusp of the Shimura variety--his original motivation to prove the Moonshine conjecture. The second is that the divisor of these modular forms are known to be a linear combination of special divisors dictated by the principal part of the input weakly holomorphic forms. The second feature has been extended to produce so-called automorphic green functions for special divisors using harmonic Maass forms via regularized theta lifting by Bruinier (\cite{jan} and Bruinier-Funke (\cite{BF}), which turned out to be very useful to generalization of the well-known Gross-Zagier formula (\cite{GZ}) and the beautiful Gross-Zagier factorization formula of singular moduli (\cite{GZSingular}) to Shimura varieties of orthogonal type $(n, 2)$ and unitary type $(n, 1)$ (see for example \cite{BY09},  \cite{BKY12},  \cite{BHY15}, \cite{AGHMPsmall}, \cite{AGHMPbig}, \cite{Schofer}, \cite{YY16}). On the other hand, the Borcherds product expansion and in particular its integral structure is essential to prove modularity of some generating functions of arithmetic divisors on these Shimura varieties (\cite{BHKRY}, \cite{HMP}). Borcherds products are also closely related to Mock theta functions (see for example \cite{Ono} and references there).

 We should mention that the analogue of the Borcherds product to unitary Shimura varieties of type $(n, 1)$ has been worked out by Hofmann (\cite{hof}). The Borcherds product expansion in the unitary case is a little more complicated as it is a Fourier-Jacobi expansion rather than Fourier expansion. The purpose of this note is to give some explicit examples of these Borcherds product expansion in concrete term. For this reason, we focus on the Picard modular surface associated to the Hermitian lattice $L=\Z[i] \oplus \Z[i] \oplus \frac{1}2 \Z[i]$ with
 Hermitian form
 $$
 \langle x,  y \rangle = x_1 \bar{y}_3 + x_3 \bar{y}_1 + x_2 \bar{y}_2.
 $$
 Our inputs are weakly holomorphic modular forms for $\Gamma_0(4)$ of weight $-1$, character $\chi_{-4}:=\left(\frac{-4}{}\right)$ which have poles only at the cusp $\infty$, which we denote by $M_{-k}^{!, \infty}(\Gamma_0(4), \chi_{-4}^k)$ with $k=1$. Our first result (Theorem \ref{basisall}) is to give a canonical basis $F_{k,m}$ ($m \ge 1$) for the infinitely dimensional vector space for every $k \ge 1$. The even $k$ case was given by Haddock and Jenkins in \cite{paul} in a slightly different fashion. Since $\Gamma_0(4)$ is normal in $\SL_2(\Z)$, a simple conjugation also gives a canonical basis for the space of weakly holomorphic forms of $\Gamma_0(4)$ with weight $-k$, character $\chi_{-4}^k$, and having poles only at cusp $0$ (resp. $\frac{1}2$).

 Next, we use standard induction procedure to produce vector valued weakly modular forms for $\SL_2(\Z)$ using our lattice $L$ which will be used to construct
 Picard modular forms on $U(2, 1)$ (described above). Although the resulting vector valued modular forms for $\SL_2(\Z)$ from the three different scalar valued spaces $M_{-k}^{!, P}(\Gamma_0(4), \chi_{-4}^k)$, $P=\infty, 0, \frac{1}2$ are linearly independent, they don't generate the whole space. We find it interesting.  This concludes Part I of our note, which  should be of independent interest.

 In Part II, we focus on the unitary group $U(2, 1)$ associated to the above Hermitian form and give explicit Borcherds product expansion of the Picard  modular forms constructed from  $F_{m}=F_{1, m}$.  The delicate part is to choose a proper Weyl chamber, which is a dimensional $3$ real manifold and described it explicitly and carefully. Our main formula is Theorem \ref{newbor}. We remark that the same method also applies to high dimensional unitary Shimura varieties of unitary type $(n, 1)$ using forms in $M_{1-n}^{!, P}(\Gamma_0(4), \chi_{-4}^k)$ where $P$ is a cusp for $\Gamma_0(4)$. We restrict to $U(2, 1)$ for being as explicit as possible.

\section{Part I: Vector Valued Modular Forms}
In this part, we derive canonical basis for the space $M_{-k}^{!,\infty}(\Gamma_{0}(4),\chi_{-4}^{k})$ for any integer $k\geq0$, and investigate the properties of the vector valued modular forms arising from $M_{-k}^{!,\infty}(\Gamma_{0}(4),\chi_{-4}^{k})$. For completeness, we will also give canonical basis for $M_{-k}^{!,0}(\Gamma_{0}(4),\chi_{-4}^{k})$ and $M_{-k}^{!,\frac{1}{2}}(\Gamma_{0}(4),\chi_{-4}^{k})$.

\subsection{Canonical Basis for $M_{-k}^{!,\infty}(\Gamma_{0}(4),\chi_{-4}^{k})$}\mbox{}\\

Let $\chi_{-4}(\cdot):=\left(\frac{-4}{\cdot}\right)$ be the Kronecker symbol modulo $4$. Recall that $X_0(4)$ has  3 cusps, represented by $\infty$, $0$, and $\frac{1}2$.  For each cusp $P$,  let $M_{-k}^{!,P}(\Gamma_{0}(4),\chi_{-4}^{k})$ denote the space of weakly holomorphic modular forms, which are holomorphic everywhere except at the cusp $P$, of weight $-k$ on $\Gamma_0(4)$ with character $\chi_{-4}^{k}$. We will focus mainly on the cusp $\infty$ and will remark on other cusps (very similar) in the end. We will also denote $M_{-k}^!(\Gamma_0(4), \chi_{-4}^k)$ for the space of weakly holomorphic modular forms for $\Gamma_0(4)$ of weight $-k$  and character $\chi_{-4}^k$.  

Let $\tau$ be a complex number with positive imaginary part, and set $q=e(\tau)=e^{2\pi i\tau}$, and $q_{r}=e^{2\pi i\tau/r}$. The Dedekind eta function is defined by
$$
\eta(\tau)=q^{1/24}\prod_{n=1}^{\infty}(1-q^{n}).
$$
Throughout this paper, we write $\eta_{m}$ for $\eta(m\tau)$.
The well known Jacobi theta functions are defined by
$$
\vartheta_{00}(\tau)=\sum_{n=-\infty}^{\infty}q^{n^{2}},\quad\vartheta_{01}(\tau)=\sum_{n=-\infty}^{\infty}(-q)^{n^{2}},\quad\vartheta_{10}(\tau)=\sum_{n=-\infty}^{\infty}q^{\left(n+\frac{1}{2}\right)^{2}}.
$$
Now we define three functions as follows.
\begin{align}
\label{y}
\theta_{1}=\theta_{1}(\tau):&=\frac{1}{16}\vartheta^{4}_{10}(\tau)=\frac{\eta_{4}^{8}}{\eta_{2}^{4}}=q+O(q^{2}),\\
\label{z}
\theta_{2}=\theta_{2}(\tau):&=\vartheta^{2}_{00}(\tau)=\frac{\eta_{2}^{10}}{\eta_{1}^{4}\eta_{4}^{4}}=1+O(q),\\
\label{j4}
\varphi_{\infty}=\varphi_{\infty}(\tau):&=\left(\frac{\eta_{1}}{\eta_{4}}\right)^{8}=q^{-1}+O(1).
\end{align}
Here are some basic facts \cite{paul} about the functions $\theta_{1}$, $\theta_{2}$ and $\varphi_{\infty}$.
\begin{enumerate}
\item{$\theta_{1}(\tau)$ is a holomorphic modular form of weight~2 on $\Gamma_{0}(4)$ with trivial character, has a simple zero at the cusp $\infty$, and vanishes nowhere else.}

\item{$\theta_{2}(\tau)$ is a holomorphic modular form of weight~1 on $\Gamma_{0}(4)$ with character $\chi_{-4}$, has a zero of order $\frac{1}{2}$ at the irregular cusp $\frac{1}{2}$, and vanishes nowhere else.}

\item{$\varphi_{\infty}(\tau)$ is a modular form of weight~0 on $\Gamma_{0}(4)$ with trivial character, has exactly one simple pole at the cusp $\infty$ and a simple zero at the cusp $0$.}
\end{enumerate}

\begin{Theorem}
\label{basisall}
\begin{enumerate}
\item For $k \ge 1 $ odd,  there is a (canonical) basis $F_{k, m}$ ($m \ge 1$ of $M^{!,\infty}_{-k}(\Gamma_{0}(4),\chi_{-4})$ whose Fourier expansion has the following property:
    $$
 F_{k, m}  = q^{-\frac{k+1}2 -m+1 } + \sum_{n \ge  -\frac{k-1}2} c(n) q^n.
$$

\item For $k >1$ even, there is a (canonical) basis $F_{k, m}$ ($m \ge 1$ of $M^{!,\infty}_{-k}(\Gamma_{0}(4))$ whose Fourier expansion has the following property:
$$
F_{k,m} =q^{-\frac{k}{2}-m+1}+\sum_{n\geq-\frac{k}{2}+1}c(n)q^{n},
$$

\end{enumerate}

\end{Theorem}

 \begin{proof} We prove (1) first.  Notice that $X_0(4)$ has no elliptic points  \cite[Section 3.9]{diamond}. For $F \in M_{-k}^{!, \infty}(\Gamma_0(4), \chi_{-4})$, the valence formula for $\Gamma_{0}(4)$ asserts that
$$
\sum_{z\in\Gamma_{0}(4)\backslash\mathbb{H}}{\rm ord}_{z}(F)+{\rm ord}_{\infty}(F)+{\rm ord}_{0}(F)+{\rm ord}_{1/2}(F)=-\frac{k}{2}.
$$
This implies ${\rm ord}_{1/2} F \ge \frac{1}2$ ($1/2$ is the unique irregular cusp), ${\rm ord}_{\infty}(F)\leq-\frac{k+1}{2}$. This implies the uniqueness of the basis $\{F_{k, m}\}$ if it exists.  We prove the existence by inductively construct a sequence of  monic polynomials $P_{k, m}(x)$ of  degree $m$ ($m \ge 0$) such that $F_{k, m+1}=\theta_2\theta_1^{-\frac{k+1}2} P_{k, m} (\varphi_\infty)$ give the basis we seek, i.e.,   with the following property
\begin{equation} \label{eq:Pkm}
F_{k, m+1}=\theta_2\theta_1^{-\frac{k+1}2} P_{k, m} (\varphi_\infty)  = q^{-\frac{k+1}2 -m } + \sum_{n  \geq -\frac{k-1}2} c(n) q^n.
\end{equation}
(The awkward notation $F_{k, m+1}$ instead of $F_{k, m}$ will be clear in last section.)

\begin{enumerate}

\item   Notice that $\theta_2\theta_1^{-\frac{k+1}2} \in  M_{-k}^{!, \infty}(\Gamma_{0}(4),\chi_{-4})$ with
$$
\theta_2\theta_1^{-\frac{k+1}2} = q^{-\frac{k+1}2} + \sum_{n \geq -\frac{k-1}2} c(n) q^n.
$$
 So we can and will first define $P_{k, 0}=1$.

\item  For $m \ge 1$, assume that $P_{k, m-1}(x)  \in  \mathbb C[x] $ is constructed  with degree $m-1$, leading coefficient $1$, and   the property
$$
F_{k, m}=\theta_2\theta_1^{-\frac{k+1}2} P_{k, m-1} (\varphi_\infty)  = q^{-\frac{k+1}2 -m +1} + \sum_{n  \geq -\frac{k-1}2} c(n) q^n.
$$
Then  it is easy to see
$$
F_{k, m} \varphi_\infty
=q^{-\frac{k+1}2-m} + \sum_{n > -\frac{k+1}2 -m} d(n) q^n.
$$
Let
$$
P_{k, m} =x P_{k, m-1} - \sum_{n=-\frac{k+1}2 -m+1}^{-\frac{k+1}2} d(n) P_{k, -n},
$$
and
$$F_{k, m+1} =\theta_2\theta_1^{-\frac{k+1}2} P_{k, m} (\varphi_\infty).
$$
Then   $F_{k, m+1}$ satisfies (\ref{eq:Pkm}). By induction, we prove the existence of the basis $\{F_{k, m}\}$, and (1).
\end{enumerate}

The proof of (2) is similar and is left to the reader. The basis $\{ F_{k,m+1}\}$, $m \ge 0$, has the form
\begin{equation}
F_{k, m+1} = \theta_{1}^{-\frac{k}{2}}Q_{k,m}(\varphi_{\infty})=q^{-\frac{k}{2}-m}+\sum_{n=-\frac{k}{2}+1}^{\infty}c(n)q^{n}
\end{equation}
for a unique monic polynomial $Q_{k, m}$ of degree $m$.

\end{proof}

\begin{Remark} The canonical basis given in Theorem \ref{basisall}(2) was given in a slightly different form first by Haddock and Jenkins \cite{paul}.
\end{Remark}

The following corollary follows directly from  the proof of  Theorem \ref{basisall}(1).
\begin{Corollary}
Every weakly holomorphic modular form $f(\tau)\in M_{-k}^{!,\infty}(\Gamma_{0}(4),\chi_{-4}^{k})$ with $k$ odd, vanishes at cusp $1/2$.
\end{Corollary}

\subsection{Vector Valued Modular Form Arising from $M_{-k}^{!,\infty}(\Gamma_{0}(4),\chi_{-4}^{k})$}\mbox{}\\

Let $L$ be an even lattice over $\mathbb{Z}$ with symmetric non-degenerate bilinear form $(\cdot,\cdot)$ and associated quadratic form $Q(x)=\frac{1}2(x, x)$. Let $L'$ be the dual lattice of $L$. Assume that $L$ has rank $2m+2$ and signature $(2m,2)$. Then the Weil representation of the metaplectic group $\Mp_{2}(\mathbb{Z})$ on the group algebra $\mathbb{C}[L'/L]$ factors through $\SL_{2}(\mathbb{Z})$. Thus we have a unitary representation $\rho_{L}$ of $\SL_{2}(\mathbb{Z})$ on $\mathbb{C}[L'/L]$, defined by
\begin{align*}
\rho_{L}(T)\phi_{\mu}&=e(-Q(\mu))\phi_{\mu},\\
\rho_{L}(S)\phi_{\mu}&=\frac{\sqrt{i}^{2m-2}}{\sqrt{|L'/L|}}\sum_{\beta\in L'/L}e((\mu,\beta))\phi_{\beta}
\end{align*}
where $T=\begin{pmatrix}1&1\\0&1\end{pmatrix}$, $S=\begin{pmatrix}0&-1\\1&0\end{pmatrix}$, $\phi_{\gamma}$ for $\mu\in L'/L$ are the standard basis elements of $\mathbb{C}[L'/L]$ and  $e(z)=e^{2\pi iz}$. We remark that the Weil representation  $\rho_L$ depends only on the finite quadratic module $(L'/L, Q)$ (called the discriminant group of $L$), where $Q(x+L) =Q(x) \pmod 1 \in \Q/Z$.

Let $k$ be an integer and $\vec F$ be a $\mathbb{C}[L'/L]$ valued function on $\mathbb{H}$ and let $\rho=\rho_{L}$ be a representation of $\SL_{2}(\mathbb{Z})$ on $\mathbb{C}[L'/L]$. For $\gamma\in \SL_{2}(\mathbb{Z})$ we define the slash operator by
$$
\left(\left.\vec F\right|_{k,\rho}\gamma\right)(\tau)=(c\tau+d)^{-k}\rho(\gamma)^{-1}\vec F(\gamma\tau),
$$
where $\gamma=\begin{pmatrix}a&b\\c&d\end{pmatrix}$ acts on $\mathbb{H}$ via $\gamma\tau=\frac{a\tau+b}{c\tau+d}$.

\begin{Definition}
Let $k$ be an integer. A function $\vec F:\mathbb{H}\to\mathbb{C}[L'/L]$ is called a weakly holomorphic vector valued modular form of weight $k$ with respect to $\rho=\rho_{L}$ if it satisfies
\begin{enumerate}
\item{$\left.\vec F\right|_{k,\rho}\gamma=F$ for all $\gamma\in \SL_{2}(\mathbb{Z})$,}
\item{$\vec F$ is holomorphic on $\mathbb{H}$,}
\item{$\vec F$ is meromorphic at the cusp $\infty$.}
\end{enumerate}
The space of such forms is denoted by $M^{!}_{k,\rho}$.
\end{Definition}
The invariance of $T$-action implies that $\vec F\in M^{!}_{k,\rho}$ has a Fourier expansion of the form
$$
\vec F=\sum_{\mu\in L'/L}\sum_{\substack{n\in\mathbb{Q}\\n\gg-\infty}}c(n,\phi_\mu)q^{n}\phi_{\mu}.
$$
Note that $c(n,\phi_\mu)=0$ unless $n\equiv-Q(\mu)\pmod{1}$.

From now on, we focus on the special case  with the discriminant group $L'/L\cong \mathbb{Z}/2\mathbb{Z}\times\mathbb{Z}/2\mathbb{Z}$ with quadratic form $Q(x,y)=\frac{1}{4}(x^{2}+y^{2})\pmod{1}$. For our purpose (in last section), it is convenient to identify
 $\mathbb{Z}/2\mathbb{Z}\times\mathbb{Z}/2\mathbb{Z} \cong \mathbb{Z}[i]/2\mathbb{Z}[i]$, where $Q(z) =\frac{1}4 z \bar z \in \Q/\Z$. We write $\phi_{0}$, $\phi_{1}$, $\phi_{i}$ and $\phi_{1+i}$ for the basis elements of $\mathbb{C}[L'/L]$ corresponding to $(0,0)$, $(1,0)$, $(0,1)$ and $(1,1)$ respectively.

Let $F=F(\tau)\in M_{-k}^{!,\infty}(\Gamma_0(4),\chi_{-4})$ with $k$ odd and positive.  Then using $\Gamma_{0}(4)$-lifting, we can construct a vector valued modular form $\vec F=\vec F(\tau)$ arising from $F(\tau)$ as follows:
\begin{equation}
\vec{F}(\tau)=\sum_{\gamma\in\Gamma_0(4)\backslash{\rm SL}_{2}(\mathbb{Z})}(\left.F\right|_{-k}\gamma)\rho_{L}(\gamma)^{-1}\phi_{0}
=\frac{1}{2}\sum_{\gamma\in\Gamma_1(4)\backslash{\rm SL}_{2}(\mathbb{Z})}(\left.F\right|_{-k}\gamma)\rho_{L}(\gamma)^{-1}\phi_{0}.
\end{equation}
Define modular forms $F_{0}$, $F_{2}$ and $F_{3}$ as follows. Let
$$
\left.F\right|_{-k}\begin{pmatrix}0&-1\\1&0\end{pmatrix}=\sum_{n=0}^{\infty}a(n)q_{4}^{n}.
$$
Then for $j\in\{0, 2,3\}$, we write
$$
F_{j}=\sum_{n=0}^{\infty}a(4n+j)q_{4}^{4n+j}.
$$
We also define $F_{1/2}$ to be
$$
F_{1/2}=\left.F\right|_{-k}\begin{pmatrix}1&0\\2&1\end{pmatrix}=\sum_{n=0}^{\infty}b(n)q_{2}^{n}.
$$
Then a simple calculation gives
\begin{align}
\label{FFF}
\vec F(\tau)
&=\left(-2iF_{0}+F\right)\phi_{0}-2iF_{3}\phi_{1}-2iF_{3}\phi_{i}+\left(-2iF_{2}-F_{1/2}\right)\phi_{1+i}.
\end{align}
The following theorem gives some basic facts about $F_{0}$, $F_{2}$, $F_{3}$ and $F_{1/2}$.
\begin{Theorem}
With the above definitions, we have
\begin{align}
\label{f00}
F_{0}&\in M_{-k}^{!}(\Gamma_{0}(4),\chi_{-4}),\\
\label{f03}
F_{3}&\in M_{-k}^{!}(\Gamma_{0}(4),\chi_{1})\\
\intertext{where $\chi_{1}(\gamma)=\chi_{-4}(d)e(-ab/4)$ for
$\gamma=\begin{pmatrix}a&b\\c&d\end{pmatrix}\in\Gamma_{0}(4)$,}
\label{f02}
(2iF_{2}+F_{1/2})&\in M^{!}_{-k}(\Gamma_{0}(4),\chi_{2})\\
\intertext{where $\chi_{2}(\gamma)=\chi_{-4}(d)e(-ab/2)$ for $\gamma=\begin{pmatrix}a&b\\c&d\end{pmatrix}\in\Gamma_{0}(4)$,}
\intertext{and}
\label{f12}
F_{1/2}&\in M_{-k}^{!}(\delta^{-1}\Gamma_{0}(4)\delta,\chi_{-4})
\intertext{where $\delta=\begin{pmatrix}1&0\\2&1\end{pmatrix}$.}\nonumber
\end{align}
\end{Theorem}
\begin{proof}
By \eqref{FFF}, and \cite[Section 3, p.\,6]{sche} or \cite[Proposition 4.5]{sche2}, we can show that for $\gamma=\begin{pmatrix}a&b\\c&d\end{pmatrix}\in\Gamma_{0}(4)$,
\begin{align}
\label{tran1}
\left.(-2iF_{0}+F)\right|_{-k}\gamma&=\chi_{-4}(d)(-2iF_{0}+F),\\
\label{tran2}
\left.F_{3}\right|_{-k}\gamma&=\chi_{-4}(d)e(-ab/4)F_{3},\\
\label{tran3}
\left.(-2iF_{2}-F_{1/2})\right|_{-k}\gamma&=\chi_{-4}(d)e(-ab/2)(-2i F_{2}-F_{1/2}).
\end{align}
Since $F\in M_{-k}^{!}(\Gamma_{0}(4),\chi_{-4})$, then \eqref{tran1} implies \eqref{f00}. Relations \eqref{f03} and \eqref{f02} follow directly from \eqref{tran2} and \eqref{tran3}, respectively. The last relation \eqref{f12} follows from the definition of $F_{1/2}$,
$$
F_{1/2}=\left.F\right|_{-k}\begin{pmatrix}1&0\\2&1\end{pmatrix}.
$$
\end{proof}

\begin{Theorem}
\label{vecF}
Let $k$ be odd. Let $F=F(\tau)\in M^{!,\infty}_{-k}(\Gamma_{0}(4),\chi_{-4})$ with
$$
F(\tau)=\sum_{n=-m}^{\infty}c(n)q^{n}.
$$
Write
$$
\left.F\right|_{-k}\begin{pmatrix}0&-1\\1&0\end{pmatrix}=\sum_{n=0}^{\infty}a(n)q_{4}^{n}\quad\mbox{and}\quad \left.F\right|_{-k}\begin{pmatrix}1&0\\2&1\end{pmatrix}=\sum_{n=0}^{\infty}b(n)q_{2}^{n}.
$$
And let the $\Gamma_{0}(4)$-lifting of $F$ be
$$
\vec F(\tau)=\sum_{\mu\in L'/L}\sum_{\substack{n\in\mathbb{Q}\\n\gg-\infty}}c(n,\phi_{\mu})q^{n}\phi_{\mu}.
$$
Then we have
\begin{enumerate}[(i)]
\item{
\begin{align*}
c(n,\phi_{0})&=-2ia(4n)+c(n),\\
c(n,\phi_{1})=c(n,\phi_{i})&=-2ia(4n),\\
c(n,\phi_{1+i})&=-2ia(4n)-b(2n),
\end{align*}
}
\item{
the principal part of the vector valued modular form $\vec F(\tau)$
is
$$
\left(c(-m)q^{-m}+\cdots+c(-1)q^{-1}\right)\phi_{0},
$$
}
\item{
the constant term of the $\phi_{0}$-component of $\vec F(\tau)$ is
$$
c(0,\phi_{0})=-(8i)^{k+1}\sum_{n=\frac{k+1}{2}}^{m}c(-n)P_{k,n-\frac{k+1}{2}}(0)+c(0).
$$
}
\end{enumerate}
In particular, when $k=1$, the constant term of the $\phi_{0}$-component of $\vec{F}(\tau)$ is
\begin{equation}
\label{k1}
c(0,\phi_{0})=\sum_{n=1}^{m}c(-n)\left(\sum_{d|n}\left(64\chi_{-4}(n/d)+4\chi_{-4}(d)\right)d^{2}\right).
\end{equation}
\end{Theorem}

\begin{proof}
Assertion (i) follows directly from \eqref{FFF}.
For the assertion (ii), since $F$ is holomorphic at 0 and $\frac{1}{2}$, then $F_{j}$ for $j\in\{0,2,3\}$ and $F_{1/2}$ will not contribute anything to the principal part of $\vec F$, and thus by \eqref{FFF} the principal part of $\vec F$ is
$$
\left(c(-m)q^{-m}+\cdots+c(-1)q^{-1}\right)\phi_{0}.
$$
For the assertion (iii), we first note by (i) that
$$
c(0,\phi_0)=-2ia(0)+c(0).
$$
By Theorem~\ref{basisall}(1), we have
\begin{equation}
\label{Ff1}
F=c(-m)\theta_{2}\theta_{1}^{-\frac{k+1}{2}}P_{k,m-\frac{k+1}{2}}(\varphi_{\infty})+\cdots+c\left(-\frac{k+1}{2}\right)\theta_{2}\theta_{1}^{-\frac{k+1}{2}}P_{k,0}(\varphi_{\infty})
\end{equation}
  We can verify that
$$
\left.\theta_{2}\theta_{1}^{-\frac{k+1}{2}}\varphi_{\infty}^{l}\right|_{-k}\begin{pmatrix}0&-1\\1&0\end{pmatrix}=O(q^{\frac{l}{4}}),
$$
and thus $\left.\theta_{2}\theta_{1}^{-\frac{k+1}{2}}\varphi_{\infty}^{l}\right|_{-1}\begin{pmatrix}0&-1\\1&0\end{pmatrix}$ will not contribute anything to the constant term of $F_{0}$ when $l\geq1$. Therefore,
\begin{align*}
a(0)&=\left(\left.\sum_{n=\frac{k+1}{2}}^{m}c(-n)P_{k,n-\frac{k+1}{2}}(0)\theta_{2}\theta_{1}^{-\frac{k+1}{2}}\right|_{-k}\begin{pmatrix}0&-1\\1&0\end{pmatrix}\right)_{0}\\
&=-(8i)^{k+1}\sum_{n=\frac{k+1}{2}}^{m}c(-n)P_{k,n-\frac{k+1}{2}}(0)
\end{align*}
where $(f)_{0}$ denote the constant term of the $q$-expansion of $f$.
Hence, we have
$$
c(0,\phi_{0})=-(8i)^{k+1}\sum_{n=\frac{k+1}{2}}^{m}c(-n)P_{k,n-\frac{k+1}{2}}(0)+c(0).
$$
For \eqref{k1},  according to (iii), we need to show that
$$
P_{1,m}(0)=\sum_{d|(m+1)}\chi_{-4}((m+1)/d)d^{2}\quad\mbox{and}\quad c(0)=\sum_{n=1}^{m}c(-n)\left(4\sum_{d|n}\chi_{-4}(d)d^{2}\right).
$$
For the first formula, we first observe that
$$
\theta_{2}\theta_{1}^{-1}\varphi_{\infty}^{\ell}=q^{-\ell-1}+\sum_{j=1}^{\ell}c_{\ell}(-j)q^{-j}+O(1)
$$
for $0\leq\ell\leq m$. Thus there are $b_{1},\ldots,b_{m-1}$ such that
$$
h(\tau):=\theta_{2}\theta_{1}^{-1}\varphi_{\infty}^{m}+b_{m-1}\theta_{2}\theta_{1}^{-1}
\varphi_{\infty}^{m-1}+\cdots+b_{1}\theta_{2}\theta_{1}^{-1}\varphi_{\infty}=q^{-m-1}+a(-1)q^{-1}+O(1)
$$
for some constant $a(-1)$. Let $g(\tau)$ be defined by
$$
g(\tau)=\sum_{n=1}^{\infty}\left(\sum_{d|n}\chi_{-4}(n/d)d^{2}\right)q^{n}=\sum_{n=1}^{\infty}d_{n}q^{n}.
$$
It is known \cite{kolberg} that $g(\tau)$ is a weight~$3$ modular form on $\Gamma_{0}(4)$ with character $\chi_{-4}$. We note by the basic facts about $\theta_{1}$, $\theta_{2}$ and $\varphi_{\infty}$ that $h(\tau)$ vanishes at the cusps $1/2$ and $0$. Then by \cite[Theorem 3.1]{bor2}, we have
$$
d_{m+1}+a(-1)=0,\,\, i.e., \,\,  d_{m+1}=-a(-1).
$$
Therefore
$$
P_{1,m}(0)=d_{m+1}=\sum_{d|(m+1)}\chi_{-4}((m+1)/d)d^{2}.
$$
This proves the first formula.
For the second one, the proof is similar by noting that
$$
h_1(\tau):=\theta_{2}\theta_{1}^{-1}P_{1,m}(\varphi_{\infty})=q^{-m-1}+C+O(q)
$$
and
$$
g_1(\tau)=1+4\sum_{n=1}^{\infty}\left(\sum_{d|n}\chi_{-4}(d)d^{2}\right)q^{n}
$$
is \cite{kolberg} a weight~3 modular form on $\Gamma_{0}(4)$ with character $\chi_{-4}$.
Then again \cite[Theorem 3.1]{bor2} shows that
$$
C=4\sum_{d|(m+1)}\chi_{-4}(d)d^{2}.
$$
This together with \eqref{Ff1} proves the second formula.
\end{proof}

\begin{Example}
Let $k=1$ and $F(\tau)=\theta_{2}\theta_{1}^{-1}=\frac{\eta_{2}^{14}}{\eta_{1}^{4}\eta_{4}^{12}}\in M_{-1}^{!,\infty}(\Gamma_{0}(4),\chi_{-4})$. Then we have
\begin{equation}
\label{Ff}
\vec F(\tau)=\left(-2iF_{0}+F\right)\phi_{0}-2iF_{3}\phi_{1}-2iF_{3}\phi_{i}+\left(-2iF_{2}-F_{1/2}\right)\phi_{1+i}
\end{equation}
where $F_{0}$, $F_{2}$, $F_{3}$ and $F_{1/2}$ are defined as follows; suppose
\begin{align*}
\left.F\right|_{-1}\begin{pmatrix}0&-1\\1&0\end{pmatrix}&=32i\frac{\eta(\tau/2)^{14}}{\eta(\tau/4)^{4}\eta(\tau)^{12}}\\
&=32i\big(1+12q^{1/4}+76q^{2/4}+352q^{3/4}+1356q+4600q^{5/4}\\
&\qquad\qquad+14176q^{6/4}+40512q^{7/4}+\cdots\big)\\
&=32i\left(1+1356q+O(q^{2})\right)\\
&\quad+32i\left(12q^{1/4}+4600q^{5/4}+O(q^{9/4})\right)\\
&\quad+32i\left(76q^{2/4}+14176q^{6/4}+O(q^{10/4})\right)\nonumber\\
&\quad+32i\left(352q^{3/4}+40512q^{7/4}+O(q^{11/4})\right),
\end{align*}
then
\begin{align*}
F_{0}&=32i\left(1+1356q+O(q^{2})\right),\\
F_{2}&=32i\left(76q^{2/4}+14176q^{6/4}+O(q^{10/4})\right)\nonumber\\
F_{3}&=32i\left(352q^{3/4}+40512q^{7/4}+O(q^{11/4})\right).
\end{align*}
 And
$$
F_{1/2}=\left.F\right|_{-1}\begin{pmatrix}1&0\\2&1\end{pmatrix}=64\left(q^{1/2}-8q^{3/2}+42q^{5/2}+O(q^{7/2})\right).
$$
From \eqref{Ff}, we note that the principal part of $F$ is $e(-\tau)\phi_{0}$ and the constant term of the $\phi_{0}$-component is $c(0,\phi_0)=68$.
\end{Example}

\subsection{Canonical Basis for $M_{-k}^{!,0}(\Gamma_{0}(4),\chi_{-4}^{k})$ and $M_{-k}^{!,\frac{1}{2}}(\Gamma_{0}(4),\chi_{-4}^{k})$}\mbox{}\\

We complete this section by giving canonical basis for the other two companions of $M^{!,\infty}_{-k}(\Gamma_{0}(4),\chi_{-4}^{k})$.

Let $\theta_{3}(\tau)$, $\varphi_{0}(\tau)$ and $\varphi_{1/2}(\tau)$ be defined by
\begin{align}
\label{t3}
\theta_{3}=\theta_{3}(\tau):&=\vartheta_{01}^{4}(\tau)=\frac{\eta_{1}^{8}}{\eta_{2}^{4}}=1+O(q),\\
\label{p0}
\varphi_{0}=\varphi_{0}(\tau):&=\left(\frac{\eta_{4}}{\eta_{1}}\right)^{8}=q+O(q^{2}),\\
\label{p1/2}
\varphi_{1/2}=\varphi_{1/2}(\tau):&=\frac{\eta_{1}^{8}\eta_{4}^{16}}{\eta_{2}^{24}}=q+O(q^{2}).
\end{align}
Here are some basic facts about $\theta_{3}$, $\varphi_{0}$ and $\varphi_{1/2}$:
\begin{enumerate}
\item{$\theta_{3}(\tau)$ is a weight~2 modular form on $\Gamma_{0}(4)$ with trivial character, has a simple zero at the cusp~0, and vanishes nowhere else;}
\item{$\varphi_{0}(\tau)$ is a weight~0 modular form on $\Gamma_{0}(4)$ with trivial character, has a simple pole at the cusp~0 and a simple zero at the cusp $\infty$, and vanishes nowhere else;}
\item{$\varphi_{1/2}(\tau)$ is a weight~0 modular form on $\Gamma_{0}(4)$ with trivial character, has a simple pole at the cusp $\frac{1}{2}$ and a simple zero at the cusp $\infty$, and vanishes nowhere else.}
\end{enumerate}

\begin{Theorem}
\label{basisall0}
Let $\theta_{2}$, $\theta_{3}$ and $\varphi_{0}$ be as defined in \eqref{z}, \eqref{t3} and \eqref{p0}, respectively.
\begin{enumerate}
\item{For $k$ odd, the set $\{\theta_{2}\theta_{3}^{-\frac{k+1}{2}}P_{k,m}(\varphi_{0})\}_{m=0}^{\infty}$ where $P_{k,m}$ is a monic polynomial of degree $m$ such that
$$
\left.\theta_{2}\theta_{3}^{-\frac{k+1}{2}}P_{k,m}(\varphi_{0})\right|_{-k}\begin{pmatrix}0&-1\\1&0\end{pmatrix}=q_{4}^{-\frac{k+1}{2}-m}+\sum_{n=-\frac{k-1}{2}}^{\infty}c(n)q_{4}^{n},
$$
is a canonical basis for $M^{!,0}_{-k}(\Gamma_{0}(4),\chi_{-4})$.
}

\item{For $k$ even, the set $\{\theta_{3}^{-\frac{k}{2}}P_{k,m}(\varphi_{0})\}_{m=0}^{\infty}$ where $P_{k,m}$ is a monic polynomial of degree $m$ such that
$$
\left.\theta_{3}^{-\frac{k}{2}}P_{k,m}(\varphi_{0})\right|_{-k}\begin{pmatrix}0&-1\\1&0\end{pmatrix}=q_{4}^{-\frac{k}{2}-m}+\sum_{n=-\frac{k}{2}+1}^{\infty}c(n)q_{4}^{n},
$$
is a canonical basis for $M^{!,0}_{-k}(\Gamma_{0}(4))$. }

\end{enumerate}

\end{Theorem}

\begin{Theorem}
\label{basis1/2}
Let $\theta_{2}$ and $\varphi_{1/2}$ be as defined in \eqref{z} and \eqref{p1/2}, respectively. Then the set $\{\theta_{2}^{-k}P_{k,m}(\varphi_{1/2})\}_{m=0}^{\infty}$ where $P_{k,m}$ is a monic polynomial of degree $m$ such that
$$
\left.\theta_{2}^{-k}P_{k,m}(\varphi_{1/2})\right|_{-k}\begin{pmatrix}1&0\\2&1\end{pmatrix}=q^{-\frac{k}{2}-m}+\sum_{n=-\frac{k}{2}+1}^{\infty}c(n)q^{n},
$$
is a canonical basis for $M^{!,\frac{1}{2}}_{-k}(\Gamma_{0}(4),\chi_{-4}^{k})$.
\end{Theorem}
Proofs of Theorems \ref{basisall0} and \ref{basis1/2} are similar to that of Theorem \ref{basisall}, so we omit the details.

\begin{Remark}
\label{rem2}
For a cusp $P$,  denote by $M^{!,P}_{-k,\rho_{L}}$ the space of vector valued modular forms induced from $M^{!,P}_{-k}(\Gamma_{0}(4),\chi_{-4}^k)$ via $\Gamma_{0}(4)$-lifting. We have,  by  \eqref{FFF},
$$
M^{!,\infty}_{-k,\rho_{L}}+M^{!,0}_{-k,\rho_{L}}+M^{!,\frac{1}{2}}_{-k,\rho_{L}}=M^{!,\infty}_{-k,\rho_{L}}\oplus M^{!,0}_{-k,\rho_{L}}\oplus M^{!,\frac{1}{2}}_{-k,\rho_{L}}.
$$
Clearly, $M^{!,\infty}_{-k,\rho_{L}}+M^{!,0}_{-k,\rho_{L}}+M^{!,\frac{1}{2}}_{-k,\rho_{L}}$ is a subspace of $M^{!}_{-k,\rho_{L}}$. In general, the former space may not be equal to the latter one. We first note that by \eqref{FFF} every vector valued modular form in $M^{!,\infty}_{-k,\rho_{L}}+M^{!,0}_{-k,\rho_{L}}+M^{!,\frac{1}{2}}_{-k,\rho_{L}}$ must have the same component functions at $\phi_{1}$ and $\phi_{i}$. We now give an example of functions in $M^{!}_{-1,\rho_{L}}$ that does not have this property. Let $F(\tau)=\theta_{2}\theta_{1}^{-1}\in M_{-1}^{!,\infty}(\Gamma_{0}(4),\chi_{-4})$. Then as above we write the $\Gamma_{0}(4)$-lifting of $F(\tau)$ as
\begin{align*}
\vec F(\tau)
&=\left(-2iF_{0}+F\right)\phi_{0}-2iF_{3}\phi_{1}-2iF_{3}\phi_{i}+\left(-2iF_{2}-F_{1/2}\right)\phi_{1+i}
\end{align*}
where
$$
F_{j}=\sum_{n=0}^{\infty}a(4n+j)q_{4}^{4n+j},
$$
$$
\left.F\right|_{-k}\begin{pmatrix}0&-1\\1&0\end{pmatrix}=\sum_{n=0}^{\infty}a(n)q_{4}^{n}
$$
and
$$
F_{1/2}=\left.F\right|_{-1}\begin{pmatrix}1&0\\2&1\end{pmatrix}.
$$
By \eqref{f03}, we know that $F_{3}(\tau)\in M^{!}_{-1}(\Gamma_{1}(4),\chi)$ where $\chi\left(\begin{pmatrix}a&b\\c&d\end{pmatrix}\right)=e(-b/4)$. Now we do $\Gamma_{1}(4)$-lifting on $F_{3}(\tau)$ against $\phi_{1}$ and get
$$
\vec{F}_{3}(\tau)=-4i{f}_{0}\phi_{0}+(2F_{3}+4if_{3})\phi_{1}+(-4if_{3}-2f_{1/2})\phi_{i}+4if_{2}\phi_{1+i}
$$
where
$$
f_{j}=\sum_{\substack{n\in\mathbb{Z}\\n\gg-\infty}}\tilde{a}(4n+j)q_{4}^{4n+j},
$$
$$
\left.F_{3}\right|_{-1}\begin{pmatrix}0&-1\\1&0\end{pmatrix}=\sum_{\substack{n\in\mathbb{Z}\\n\gg-\infty}}\tilde{a}(n)q_{4}^{n}
$$
 and
$$
f_{1/2}=\left.F_{3}\right|_{-1}\begin{pmatrix}1&0\\2&1\end{pmatrix}.
$$
Now the component functions at $\phi_{1}$ and $\phi_{i}$ are $2F_{3}+4if_{3}$ and $-4if_{3}-2f_{1/2}$, respectively. We can compute and verify that they are not the same. Therefore, $\vec{F}_{3}(\tau)$ is not in the space $M^{!,\infty}_{-k,\rho_{L}}+M^{!,0}_{-k,\rho_{L}}+M^{!,\frac{1}{2}}_{-k,\rho_{L}}$.
\end{Remark}

\section{Part II: Borcherds Products on $U(2,1)$}

It is well-known that the vector valued weakly modular forms construction  in Part I can be used to construct memomophic modular forms on  Shimura varieties
of orthogonal type $(n, 2)$ and unitary type $(n, 1)$ with Borcherds product formula and known divisors. In this part, we focus on one special case to make it very explicitly---the Picard modular surfaces over $\kay =\mathbb Q(i)$. In  particular, we describe a Weyl chamber explicit and write down the Borcherds product expression concretely.

This part is devoted to deriving Borcherds products lifted from a vector valued modular form arising from $M^{!,\infty}_{-1}(\Gamma_{0}(4),\chi_{-4})$.

\subsection{Picard modular surfaces over $\kay =\mathbb Q(i)$}

Let $(V, \langle \, , \rangle)$ be a Hermitian vector space over $\kay$ of signature $(2, 1)$ and let $H=U(V)$. Let $V_\mathbb C=V\otimes_\kay\mathbb C$, and
$$
\mathcal L=\{ w \in V_\mathbb C:\,  \langle w, w \rangle < 0\}.
$$
Then $\mathcal K =\mathcal L/\mathbb C^\times$ is the Hermitian domain for $H(\mathbb R)$, and $\mathcal L$ is the tautological line bundle over $\mathcal K$. For a congruence subgroup $\Gamma$ of $H(\mathbb Q)$, the associated Picard modular surface $X_\Gamma =\Gamma \backslash \mathcal K$ is defined over some number filed.

Given an isotropic line $\kay \ell$ (i.e., a cusp), choose another isotropic element $\ell'$ with $(\ell, \ell')\ne 0$.  Let $V_0 =(\kay \ell + \kay \ell')^\perp$, and let
$$
\mathcal H=\mathcal H_{\ell, \ell'}=\{ (\tau, \sigma)\in \mathbb H\times V_{0, \mathbb C}|\,  \Im{\tau}>\frac{\langle\sigma,\sigma\rangle}{4|\langle\ell',\ell\rangle|^{2}}\}.
$$
Then the map
\begin{equation}
\mathcal H \rightarrow \mathcal L, \,\, (\tau, \sigma) \mapsto z(\tau, \sigma) = 2i\langle\ell',\ell\rangle\tau\ell+\sigma + \ell'
\end{equation}
gives rise to an isomorphism $\mathcal H \cong \mathcal K$. It is also a nowhere vanishing section of the line bundle $\mathcal L$.  Using this map, we can define action of $H(\mathbb R)$ on $\mathcal H$ and automorphy factor $j(\gamma, \tau, \sigma)$ via the equation
\begin{equation} \label{eq:Action}
\gamma z(\tau, \sigma) =j(\gamma, \tau, \sigma) z(\gamma(\tau, \sigma)).
\end{equation}
\begin{Definition}
Let $\Gamma$ be a unitary modular group. A holomorphic automorphic form of weight $k$ and with character $\chi$ for $\Gamma$ is a function $g:\mathcal{H}\to\mathbb{C}$, with the following properties:
\begin{enumerate}
\item{$g$ is holomorphic on $\mathcal{H}$,}
\item{$g(\gamma(\tau,\sigma))=j(\gamma;\tau,\sigma)^{k}\chi(\gamma)g(\tau,\sigma)$ for all $\gamma\in\Gamma$.}
\end{enumerate}
\end{Definition}
We remark that a holomorphic modular form $g$ for $\Gamma$ is automatically holomorphic at cusps.

Now we make everything concrete and explicit. First choose a basis $\{{\bf e}_1, {\bf e}_2, {\bf e}_3\}$ of $V$ with Gram matrix
$$
J=\begin{pmatrix} 0 &0 &1
\\
0 &1 &0
\\
1 &0 &0
\end{pmatrix}
$$
so $V=\oplus \kay {\bf e}_i \cong \kay^3$ with Hermitian form
\begin{equation}
\label{herform}
\langle {x}, {y}\rangle=x_{1}\bar{y}_{3}+x_{2}\bar{y}_{2}+x_{3}\bar{y}_{1}={}^{t}xJ\bar{y},
\end{equation}
and
$$
H=H(\mathbb Q)=\{ h \in  \hbox{GL}_3(\kay)| \, h J \, {}^t\bar{h}= J\}.
$$
We take the lattice
$$
L= \mathbb{Z}[i]\oplus\mathbb{Z}[i]\oplus\frac{1}{2}\mathbb{Z}[i]
$$
(instead of the typical $\mathbb Z[i]^3$). Its $\mathbb Z$-dual lattice is
$$
L'=\{ v \in V|\,  \mbox{Tr}_{\kay/\mathbb Q} \langle v, L\rangle \subset \mathbb Z\} =\mathbb{Z}[i]\oplus \frac{1}2\mathbb{Z}[i]\oplus\frac{1}{2}\mathbb{Z}[i]
$$
So $L'/L \cong \frac{1}2\mathbb Z[i]/\mathbb Z[i]$ with quadratic form $Q(x) =x \bar x \in  \frac{1}4\mathbb Z/\mathbb Z$, which  is the same finite quadratic module considered in Part I. Let
\begin{align*}
U(L)&=\{g\in H|gL=L\}\\
&=H\cap\left \{\begin{pmatrix}\mathbb{Z}[i]&\mathbb{Z}[i]&2\mathbb{Z}[i]\\\mathbb{Z}[i]&\mathbb{Z}[i]&2\mathbb{Z}[i]\\\frac{1}{2}\mathbb{Z}[i]&\frac{1}{2}\mathbb{Z}[i]&\mathbb{Z}[i]\end{pmatrix}\right\}.
\end{align*}
be the stabilizer of $L$ in $H$, and $\Gamma_L$ be the subgroup of $U(L)$ which acts on  the discriminant group $L'/L$ trivially:
$$
\Gamma_{L}=U(L)\cap\left\{\begin{pmatrix}\mathbb{Z}[i]&2\mathbb{Z}[i]&2\mathbb{Z}[i]\\\mathbb{Z}[i]&1+2\mathbb{Z}[i]&2\mathbb{Z}[i]\\\mathbb{Z}[i]&2\mathbb{Z}[i]&\mathbb{Z}[i]\end{pmatrix}\right\}.
$$
Take the cusp $\ell={\bf e}_1$ and $\ell'={\bf e}_3$. Then $V_0 \cong \kay $ with Hermitian form $\langle x, y\rangle =x \bar y$, and
$$
\mathcal H=\{ (\tau, \sigma) \in \mathbb H \times \mathbb C| \,   4 \mbox{Im}(\tau) > |\sigma|^2\}.
$$
Moreover, one has  for $\gamma =(a_{ij}) \in H$
$$
\gamma (\tau,\sigma)=\left(\frac{a_{11}\tau+(2i)^{-1}a_{12}\sigma+(2i)^{-1}a_{13}}{2i a_{31}\tau+a_{32}\sigma +a_{33}},\frac{2i a_{21}\tau+a_{22}\sigma +a_{23}}{2i a_{31}\tau+a_{32}\sigma+a_{33}}\right).
$$
and
$$
j(\gamma, \tau,\sigma)=\frac{\langle\gamma z,\ell\rangle}{\langle\ell',\ell\rangle}=2i\tau a_{31}+a_{32}\sigma+a_{33}.
$$

Let $P_\ell$ be the stabilizer of the cusp $\kay \ell$ in $H$. Then $P_\ell=N_\ell M_\ell$ with
\begin{align*}
M_\ell&=\{ m(a, b)= \mbox{Diag}(a, b, \bar a^{-1})| \,  a \in \kay^\times, b \in \kay^1\},
\\
N_\ell&=\{ n(b, c) =\begin{pmatrix}
  1 &-2\bar b & -2b \bar b + 2i c
  \\
  0  &1 & 2b\\
  0 &0 &1
  \end{pmatrix} |\,  b \in \kay, c \in \mathbb Q
  \},
\end{align*}
where $\kay^1=\{ a \in \kay| a \bar a=1\}$ is the norm one group.  Notice that $N_\ell$ is  a Heisenberg group actin on $\mathcal H_{\ell, \ell'}$ via
$$
n(b, c) (\tau, \sigma) = ( \tau + c+  i\bar b (\sigma+b), \sigma+b).
$$
In particular
$$
n(0, c)(\tau, \sigma)= (\tau+c, \sigma).
$$
Let
$$
\Gamma_{L, \ell} =\Gamma_L \cap N_\ell = \{ n(b, c):\,  b \in  Z[i], c \in \mathbb Z\}.
$$
Then for a holomorphic modular form $f(\tau, \sigma)$ for $\Gamma_L$, we have partial Fourier expansion at the cusp $\kay\ell$:
\begin{equation}
f(\tau, \sigma)=\sum_{n \ge 0} f_n(\sigma) q^n.  
\end{equation}

\subsection{The Hermitian Space $V$ as a Quadratic Space} As mentioned in the previous subsection, the hermitian space $V$ can be viewed as a quadratic space $V_{\mathbb{Q}}$ of signature (4,2) associated with bilinear form induced from the hermitian form:
 $$
 (x, y) =\hbox{Tr}_{\kay/\mathbb Q} \langle x, y \rangle.
 $$

 Then the lattice $L$ can be considered as a quadratic $\mathbb{Z}$-lattice in $V_{\mathbb{Q}}$. Denote by
 $$
 \SO(V_{\mathbb{Q}})=\{g\in \SL(V_{\mathbb{Q}})|(gx,gy)=(x,y)\mbox{\,for all $x,y\in V_{\mathbb{Q}}$}\}
 $$
  the special orthogonal group of $V_{\mathbb{Q}}$ and its set of real points as $\SO(V_{\mathbb{Q}})(\mathbb{R})\cong \SO(4,2)$. A model for the symmetric domain of $\SO(V_{\mathbb{Q}})(\mathbb{R})$ is the Grassmannian of two-dimensional negative definite subspaces of $V_{\mathbb{Q}}$, denoted by $Gr_{O}$. It can be realized as a tube domain $\mathcal{H}_{O}$ as follows. Denote by $V_{\mathbb{Q}}(\mathbb{C})$ the complex quadratic space $V_{\mathbb{Q}}\otimes_{\mathbb{Q}}\mathbb{C}$ with $(\cdot,\cdot)$ extended to a $\mathbb{C}$-valued bilinear form.

Now we view $L$ as a $\mathbb{Z}$-lattice. Let $e_{1}\in L$ be a primitive isotropic lattice vector and choose an isotropic dual vector $e_{2}\in L'$ with $(e_{1},e_{2})=1$. Denote by $K$ the Lorentzian $\mathbb{Z}$-sublattice $K=L\cap e_{1}^{\perp}\cap e_{2}^{\perp}$ with respect to $(\cdot,\cdot)$. The tube domain model $\mathcal{H}_{O}$ is one of the the two connected components of the following subset of $K\otimes_{\mathbb{Z}}\mathbb{C}$
$$
\{Z=X+iY|X,Y\in K\otimes_{\mathbb{Z}}\mathbb{R},\,Q(Y)<0\}.
$$

Recall that $\ell={\bf e}_{1}$ and $\ell'={\bf e}_{3}$. We define
$$
e_{1}=\ell,\,e_{2}=\hat{\frac{1}{2}}\ell',\,e_{3}=-\hat{i}\ell,\,e_{4}={-\frac{\hat i}{2}}\ell'
$$
where we denote by $\hat{\mu}$ the endomorphism of $V_{\mathbb{Q}}(\mathbb{R})$ induced from the scalar multiplication with $\mu$. Then we can check that $\{e_{1},e_{2},e_{3},e_{4}\}$ is a basis for $(\mathbb{Z}[i]\ell+\mathbb{Z}[i]\ell')\otimes_{\mathbb{Z}}\mathbb{Q}$ and we can see that $K\otimes_{\mathbb{Z}}\mathbb{R}=((\mathbb{Q}e_{3}+\mathbb{Q}e_{4})\otimes_{\mathbb{Z}}\mathbb{R})\oplus (V_{0}\otimes_{\mathbb{Z}}\mathbb{R})$. Thus we can identify $Y$ with $y_{1}e_{3}+y_{2}e_{4}+\sigma\in K\otimes_{\mathbb{Z}}\mathbb{R}$. Now denote by $\mathcal{C}$ the set of $Y=y_{1}e_{3}+y_{2}e_{4}+\sigma$ with $y_{1}y_{2}+Q(\sigma)<0$, $y_{1}<0$ and $y_{2}>0$. We can fix $\mathcal{H}_{O}$ as the component for which $Y\in \mathcal{C}$. Therefore, $\mathcal{H}_{O}=K\otimes_{\mathbb{Z}}\mathbb{R}+i\mathcal{C}$.

In addition, the tube domain $\mathcal{H}_{O}$ can be mapped biholomorphically to any one of the two connected components of a negative cone of $\mathbb{P}^{1}(V_{\mathbb{Q}})(\mathbb{C})$ given by
$$
\{[Z_{L}]|(Z_{L},Z_{L})=0,\, (Z_{L},\bar{Z}_{L})<0\}.
$$
We fix this component and denote it by $\mathcal{K}_{O}$. For each $[Z_{L}]$, we can uniquely represent it as
$$
Z_{L}=e_{2}-q(Z)e_{1}+Z
$$
with $Z\in \mathcal{H}_{O}$.

\subsection{Embedding of $\mathcal{H}$ into $\mathcal{H}_{O}$} As in \cite[Section 4]{hof}, we can embed $\mathcal{H}$ into $\mathcal{H}_{O}$ via
\begin{equation}
(\tau,\sigma)\to  \iota(\tau, \sigma)= -\tau e_{3}+ie_{4}+\mathfrak{z}(\sigma)
\end{equation}
where
\begin{equation}
\mathfrak{z}(\sigma)=\frac{\hat{1}}{2}\sigma+i\left(-\frac{\hat{i}}{2}\right)\sigma.
\end{equation}
Similarly, $\mathcal{K}_{U}$ can be embedded into $\mathcal{K}_{O}$ through the identifications between $\mathcal{K}_{U}$ and $\mathcal{H}$, and between $\mathcal{K}_{O}$ and $\mathcal{H}_{O}$. Namely,
\begin{equation}
z=\ell'+2i\tau\ell+\sigma\to Z_{L}=-i\tau e_{1}+e_{2}-\tau e_{3}+ie_{4}+\mathfrak{z}(\sigma).
\end{equation}

\subsection{Weyl Chambers of $K\otimes_{\mathbb{Z}}\mathbb{R}$} In Theorem \ref{basisall} (1), we have shown that $F_{1, m} =q^{-m}+O(1)$ for $m\geq 1$, form a canonical basis for $M^{!,\infty}_{-1}(\Gamma_{0}(4),\chi_{-4})$. Therefore, to study the Borcherds product lifted from $M^{!,\infty}_{-1,\rho_{L}}$, it suffices to start with $F_{1, m}$. Since we only deal with weight $-1$ in the rest of this paper, we will simply write $F_m=F_{1, m}$, and $\vec F_m=\vec F_{1, m}$

For general definitions of the following, we refer the reader to \cite[Chapter 3.1]{jan}. For $\kappa\in K$ with $q(\kappa) >0$,  denote by $\kappa^{\perp}$ the orthogonal complement of $\kappa$ in $K\otimes_{\mathbb{Z}}\mathbb{R}$. Denote by $\mathcal{D}_{K}$ the Grassmannian of negative 1-lines of $K\otimes_{\mathbb{Z}}\mathbb{R}$, which can be realized as
\begin{align*}
\mathcal{D}_{K}
&=\{ \mathbb R  w \subset K_{\mathbb R}|,  q(w) <0\}
\\
&\cong \{w=y_{1}e_{3}+e_{4}+(y_{3}+iy_{4})| y_i \in \mathbb{R},  q(w)<0\}��
\end{align*}
Then by considering the Grassmannian  of negative 1-lines of $\kappa^{\perp}$, it corresponds to a codimension~1 sub-manifold of the Grassmannian $\mathcal{D}_{K}$ of $K\otimes_{\mathbb{Z}}\mathbb{R}$.

In our case, a Heegner divisor of index $(m,0)$, $H_{K}(m,0)$, is a locally finite union of codimension~1 sub-manifolds of $\mathcal{D}_{K}$, namely,
$$
H_{K}(m,0)=\{ z \in  \mathcal D_K| \exists \kappa \in K \hbox{ with } q(\kappa)=m \hbox{ and } (z, \kappa) =0 \}
$$
Let $\vec{F}_{m}(\tau)$ be the vector valued modular form arising from $F_m$. It is known by Theorem \ref{vecF} that the principal part of $\vec{F}_{m}(\tau)$ is $q^{-m}\phi_{0}$. The Weyl chambers attached to $\vec{F}_{m}(\tau)$  are the connected components $W_{m}$ of
$$
\mathcal{D}_{K}-H_{K}(m,0).
$$
Fix a Weyl chamber $W_{m}$ of $\mathcal{D}_{K}$, we can also define the corresponding Weyl chambers of $K\otimes_{\mathbb{Z}}\mathbb{R}$ and $\mathcal{H}$ by
\begin{align*}
W_{m,K}&=\{w\in K\otimes_{\mathbb{Z}}\mathbb{R}|\,\,\mathbb{R}w\in W_{m}\},\\
W_{m,U}&=\left\{(\tau,\sigma)\in\mathcal{H}|\Im(\iota(\tau, \sigma)) =-\Im \tau e_{3}+e_{4}-\frac{\hat{i}}{2}\sigma\in W_{m,K}\right\},
\end{align*}
respectively. In the following lemma, we give explicit description of the Weyl chamber that we use to construct Borcherds product in Theorem \ref{newbor}.

\begin{Lemma}
\begin{enumerate}[(1)]
\item{
Let
\begin{align}
\label{wm}
W_{m}&=\left\{y_{1}e_{3}+e_{4}+(y_{3}+iy_{4})\in\mathcal{D}_{K}\left|\substack{y_{1}<r^{2}+s^{2}-m+2ry_{3}+2sy_{4}\,\,\forall \,r,\,s\in\mathbb{Z},\\\\ 1+2ty_{3}+2hy_{4}>0\,\,\forall\,t,\,h\in\mathbb{Z},\,t^{2}+h^{2}=m,\\\\
ty_{3}+hy_{4}>0,\,\,\forall\,t,\,h\in\mathbb{Z},\,t^{2}+h^{2}=m,\,t>0,\\\\
y_{4}>0\,\,\text{if $m$ is a square.}
}\right.\right\}\\\nonumber\\
&\subset \left\{y_{1}e_{3}+e_{4}+(y_{3}+iy_{4})\in\mathcal{D}_{K}\left|\substack{k_{2}y_{1}<-k_{1}+2k_{3}y_{3}+2k_{4}y_{4}\,\,\forall \,k_{i}\in\mathbb{Z},\,k_{2}>0,\,k_{1}k_{2}+k_{3}^{2}+k_{4}^{2}=m,\\\\ k+2ty_{3}+2hy_{4}>0\,\,\forall\,k.\,t,\,h\in\mathbb{Z},\,k>0,\,t^{2}+h^{2}=m,\\\\
ty_{3}+hy_{4}>0,\,\,\forall\,t,\,h\in\mathbb{Z},\,t^{2}+h^{2}=m,\,t>0,\\\\
y_{4}>0\,\,\text{if $m$ is a square}
}\right.\right\}.\nonumber
\end{align}
Then $W_{m}$ is a Weyl chamber containing $e_{3}$.}

\item{
Let
$$
K_{m}=\left\{\lambda=\lambda_{1}e_{3}-\lambda_{2}e_{4}+\frac{1}{2}(\lambda_{3}+i\lambda_{4})\in K'\left|\substack{\,Q(\lambda)=m\,\,and\,\,\lambda_{3},\lambda_{4}\in2\mathbb{Z},\\\\\,or\,\,Q(\lambda)\leq0,\\\\(\lambda,W_{m})>0}\right.\right\}
$$
where $(\lambda,W_{m})>0$ meas that $(\lambda,w)>0$ for all $w\in W_{m}$.
Then
$$
K_{m}=\left\{\lambda=\lambda_{1}e_{3}-\lambda_{2}e_{4}+\frac{1}{2}(\lambda_{3}+i\lambda_{4})\left|\substack{\lambda_{1},\,\lambda_{2},\,\lambda_{3},\,\lambda_{4}\in\mathbb{Z}\\\lambda_{2}>0,\\{or\,\, \lambda_{2}=0\,\, and\,\, \lambda_{1}>0,}\\{or \,\,\lambda_{2}=\lambda_{1}=0 \,\,and\,\, \lambda_{3}>0,}\\ {or\,\,\lambda_{2}=\lambda_{1}=\lambda_{3}=0\,\, and \,\, \lambda_{4}>0}}\right.\right\}.
$$
}
\end{enumerate}
\end{Lemma}

\begin{proof}
For Assertion (1), it is clear that $W_{m}$ contains $e_{3}$ since the set of $(y_{3},y_{4},y_{1})$ determined by the inequalities in $W_{m}$ contains $y_{1}=-\infty$. We only need to show $W_{m}$ is actually a Weyl chamber.

  Write $\kappa =\kappa_1 e_3 + \kappa_2 e_4 + \kappa_3 + i \kappa_4 \in K$ with $\kappa_i \in \mathbb Z$.
  Since $(-\kappa)^{\perp}=\kappa^\perp$, we can assume  $k_{2}\geq0$. By the definition of Weyl chamber $W_{m}$, we can see that a Weyl chamber $W_{m}$ can be viewed as a connected component of $\mathbb{R}^{3}$ cut out by the planes
  $$
  k_{2}y_{1}+k_{1}+2k_{3}y_{3}+2k_{4}y_{4}=0
  $$
for all $k_{1},\ldots,k_{4}\in\mathbb{Z}$ with $k_{2}\geq0$ and $k_{1}k_{2}+k_{3}^{2}+k_{4}^{2}=m$.

 When $k_{2}=0$ and $m$ is representable by sum of two squares, then we have planes
$$
k_{1}+2k_{3}y_{3}+2k_{4}y_{4}=0$$
perpendicularly passing through the $y_{3}-y_{4}$ plane. In this case, the connected components are determined by the connected components of the $y_{3}-y_{4}$ plane cut out by the lines $$
k_{1}+2k_{3}y_{3}+2k_{4}y_{4}=0,$$ and it is easy to find that one of the connected components $\mathcal{C}_{1}$ can be identified as
$$
\left\{(y_{3},y_{4})\in\mathbb{R}^{2}\left|\substack{ 1+2ty_{3}+2hy_{4}>0\,\,\forall\,t,\,h\in\mathbb{Z},\,t^{2}+h^{2}=m,\\\\
ty_{3}+hy_{4}>0,\,\,\forall\,t,\,h\in\mathbb{Z},\,t^{2}+h^{2}=m,\,t>0,\\\\
y_{4}>0\,\,\text{if $m$ is a square}
}\right.\right\}
$$
which is a subset of
$$
\left\{(y_{3},y_{4})\in\mathbb{R}^{2}\left|\substack{ k+2ty_{3}+2hy_{4}>0\,\,\forall\,k.\,t,\,h\in\mathbb{Z},\,k>0,\,t^{2}+h^{2}=m,\\\\
ty_{3}+hy_{4}>0,\,\,\forall\,t,\,h\in\mathbb{Z},\,t^{2}+h^{2}=m,\,t>0,\\\\
y_{4}>0\,\,\text{if $m$ is a square}
}\right.\right\}.
$$
When $k_{2}>0$, with the aid of MAPLE, we can check that there is a connected component $\mathcal{C}_{2}$ of $\mathbb{R}^{3}$ covered by
$$
y_{1}=r^{2}+s^{2}-m+2ry_{3}+2sy_{4}
$$
for $r,\,s\in\mathbb{Z}$. Such a connected component contains $y_{1}<-m$, and all the other planes
$$
k_{2}y_{1}=-k_{1}+2k_{3}y_{3}+2k_{4}y_{4}
$$
for $k_{1},\ldots,k_{4}\in\mathbb{Z}$ with $k_{2}>0$ and $k_{1}k_{2}+k_{3}^{2}+k_{4}^{2}=m$. In conclusion, $W_{m}=\mathcal{C}_{1}\cap\mathcal{C}_{2}$ is a connected component of $\mathbb{R}^{3}$ cut out by the planes $$k_{2}y_{1}+k_{1}+2k_{3}y_{3}+2k_{4}y_{4}=0$$
for all $k_{1},\ldots,k_{4}\in\mathbb{Z}$ with $k_{2}\geq0$ and $k_{1}k_{2}+k_{3}^{2}+k_{4}^{2}=m$, and thus $W_{m}$ is a Weyl chamber.

Now let us prove Assertion (2).
\begin{enumerate}[(i)]
\item{Suppose that $Q(\lambda)=m$ and $\lambda_{3},\,\lambda_{4}\in2\mathbb{Z}$.
By \eqref{wm}, we note that $y_{1}e_{3}+e_{4}+(y_{3}+iy_{4})\in W_{m}$ implies that
$$k_{2}y_{1}<-k_{1}+2k_{3}y_{3}+2k_{4}y_{4}$$
for all $k_{i}\in\mathbb{Z}$ with $k_{2}>0$ and $k_{1}k_{2}+k_{3}^{2}+k_{4}^{2}=m$, which is equivalent to that $$k_{2}y_{1}+k_{1}+2k_{3}y_{3}+2k_{4}y_{4}>0$$ for all $k_{i}\in\mathbb{Z}$ with $k_{2}<0$ and $k_{1}k_{2}+k_{3}^{2}+k_{4}^{2}=m$. Therefore, when $\lambda_{2}\ne0$ and $Q(\lambda)=m$, that is, $\lambda_{1}(-\lambda_{2})+\frac{1}{4}(\lambda_{3}^{2}+\lambda_{4}^{2})=m$, $(\lambda,W_{m})>0$ if and only if $-\lambda_{2}<0$, that is, $\lambda_{2}>0$. Similarly, by the other conditions given in \eqref{wm}, we can conclude that when $Q(\lambda)=m$, $(\lambda,W_{m})>0$ if and only if $\lambda_{2}<0$, or $\lambda_{2}=0$ and $\lambda_{1}>0$, or $\lambda_{2}=\lambda_{1}=0$ and $\lambda_{3}>0$, or $\lambda_{2}=\lambda_{1}=\lambda_{3}=0$ and $\lambda_{4}>0$.
}

\item{Now suppose that $Q(\lambda)\leq0$, that is, $\lambda_{1}\lambda_{2}+\frac{1}{4}(\lambda_{3}^{2}+\lambda_{4}^{2})\leq0$. By \eqref{wm}, we know that
$$
y_{1}<r^{2}+s^{2}-m+2ry_{3}+2sy_{4}
$$
for all $r,\,s\in\mathbb{Z}$. By \cite[Lemma 3.2]{jan}, it is known that if $(\lambda,w_{0})>0$ for a $w_{0}\in W_{m}$, then $(\lambda,W_{m})>0$. Thus $(\lambda,W_{m})>0$ if and only if $\lambda_{2}>0$. When $\lambda_{2}=0$, since $Q(\lambda)\leq0$, then $\lambda_{3}=\lambda_{4}=0$, and thus $(\lambda,w)=\lambda_{1}$ for $w\in W_{m}$. This implies that $(\lambda,W_{m})>0$ if and only if $\lambda_{1}>0$ when $\lambda_{2}=0$.
}
\end{enumerate}
\end{proof}

\subsection{The Weyl Vector for $\vec{F}_{m}$}\label{wv} For the Weyl chamber $W_{m}$ described in \eqref{wm}, using \cite[Theorem 13.3]{bor} (see also \cite[Theorem 2.2]{YY16} for minor correction) or \cite[Theorem 3.22]{jan} , we can compute the corresponding Weyl vector of $\vec{F}_{m}$:
$$
\rho(W,\vec F_{m})=\rho_{e_{3}}e_{3}+\rho_{e_{4}}e_{4}+\rho
$$
with
\begin{align*}
\rho_{e_{3}}&=-\frac{1}{6}\sum_{d|m}\left(16\chi_{-4}(m/d)+\chi_{-4}(d)\right)d^{2}-\frac{1}{24}\sigma_{\chi_{-4}}(m),\\
\rho_{e_{4}}&=\frac{1}{6}\Bigg[\sigma_{\chi_{-4}}(m)-6\sigma_{1}(m)-24\left(\sum_{\substack{k+l=m\\k,l\geq1}}\sigma_{\chi_{-4}}(k)\sigma_{1}(l)\right)\\
&\qquad\qquad+\sum_{d|m}\left(16\chi_{-4}(m/d)+\chi_{-4}(d)\right)d^{2}\Bigg],\\
\rho&=-(0,\sum_{n=1}^{l}t_{n},0)
\end{align*}
where $\sigma_{\chi_{-4}}(m)=\sum_{d|m}\chi_{-4}(d)$, $0\leq t_{n}$ is the $t$-component of the integral solution of $t^{2}+h^{2}=m$, $l=\sigma_{\chi_{-4}}(m)$.

\subsection{Heegner Divisors for $\Gamma_{L}$} Let $\lambda\in L'$ be a lattice vector with positive norm, i.e., $\langle\lambda,\lambda\rangle>0$. The complement of $\lambda$ in $\mathcal{K}_{U}$ is a closed analytic subset of comdimension~1, which we denote as follows.
$$
{\bf H}(\lambda)=\{[z]\in\mathcal{K}_{U}|\langle z,\lambda\rangle=0\}.
$$
By identification between $\mathcal{K}_{U}$ and $\mathcal{H}$, ${\bf H}(\lambda)$ can also be considered as a closed analytic subset of $\mathcal{H}$, and we call such set a prime Heegner divisor on $\mathcal{H}$. Given $\beta\in L'/L$ and $m\in\mathbb{Z}_{>0}$, a Heegner divisor of index $(m,\beta)$  in  $\mathcal H$ is defined as the locally finite sum
$$
{\bf H}(m,\beta)=\sum_{\substack{\lambda\in\beta+L\\Q(\lambda)=m}}{\bf H}(\lambda).
$$
The associated Heegner divisor in $X_{\Gamma_L} =\Gamma_L \backslash \mathcal H$ is ${\bf Z}(m, \beta) = \Gamma_L \backslash {\bf H}(m,\beta)$.

\subsection{Borcherds Products} In this section, we give a family of new Borcherds products explicitly by using the results of Hofmann \cite[Thm. 4, Thm. 5 and Cor. 1]{hof}. We first summarize Hofmann's results as follows.
\begin{Theorem}[Hofmann]
\label{hofthm}
Let $\mathbb{F}$ be an imaginary quadratic field. Let $L$ be an even hermitian lattice of signature $(m,1)$ with $m\geq1$, and $\ell\in L$ a primitive isotropic vector. Let $\ell'\in L'$ an isotropic vector with $\langle\ell,\ell'\rangle\ne0$. Further assume that $L$ is the direct sum of a hyperbolic plane $H\cong \mathcal{O}_{\mathbb{F}}\oplus \partial^{-1}_{\mathbb{F}}$ and a definite part $D$ with $\langle D,H\rangle=0$.

Given a weakly holomorphic form of weight $f\in M_{1-m,\rho_{L}}^{!}$ with Fourier coefficients $c(n,\beta)$ satisfying $c(n,\beta)\in\mathbb{Z}$ for $n<0$, there is a meromorphic function $\Psi(\tau,\sigma;f)$ on $\mathcal{H}$ with the following properties:

\begin{enumerate}
\item{$\Psi(\tau,\sigma;f)$ is an automorphic form of weight $c(0,\phi_{0})/2$ for $\Gamma_{L}$ with some multiplier system $\chi$ of finite order.}

\item{The zeros and poles of $\Psi(\tau,\sigma;f)$ lie on Heegner divisors. The divisor of $\Psi(\tau,\sigma;f)$ on $X_{\Gamma_{L}}=\Gamma_{L}\backslash \mathcal{H}$ is given by
$$
div(\Psi(\tau,\sigma;f))=\frac{1}{2}\sum_{\beta\in L'/L}\sum_{\substack{n\in\mathbb{Z}-Q(\beta)\\n>0}}c(-n,\phi_{\beta}){\bf H}(n,\beta).
$$
The multiplicities of ${\bf H}(n,\beta)$ are $2$ if $2\beta=0$ in $L'/L$, and $1$ otherwise.}

\item{For a Weyl chamber $W$ whose closure containing the cusp $\mathbb{Q}e_{3}$, $\Psi(\tau,\sigma;f)$ has an infinite product expansion of the form
$$
\Psi(\tau,\sigma;f)=Ce\left(\frac{\langle z,\rho(W,f)\rangle}{\langle\ell,\ell'\rangle}\right)\prod_{\substack{\lambda\in K'\\(\lambda,W)>0}}\left[1-e\left(\frac{\langle z,\lambda\rangle}{\langle\ell,\ell'\rangle}\right)\right]^{c(-Q(\lambda),\lambda)},
$$
where $z=z(\tau,\sigma)=\ell'+\delta\langle\ell,\ell'\rangle\tau\ell+\sigma$, $\delta$ is the square root of the discriminant of $\mathbb{F}$, the constant $C$ has absolute value $1$ and $\rho(W,f)$ is the Weyl vector attached to $W$ and $f$.
}

\item{The lifting is multiplicative: $\Psi(\tau,\sigma;f+g)=\Psi(\tau,\sigma;f)\Psi(\tau,\sigma;g)$.}
\item{Let $W$ be a Weyl chamber such that the cusp corresponding to $\ell$ is contained in the closure of $W$. If this cusp is neither a pole nor a zero of $\Psi(\tau,\sigma;f)$, then we have
$$
\lim_{\tau\to \infty}\Psi(\tau,\sigma;f)=Ce\left(\bar{\rho(W,f)_{\ell}}\right)\prod_{\substack{\lambda\in K'\\\lambda=\frac{1}{2}\kappa\delta\ell\\\kappa\in\mathbb{Q}_{>0}}}\left(1-e\left(-\frac{1}{2}\kappa\bar{\delta}\right)\right)^{c(0,\lambda)}
$$
where $\rho(W,f)_{\ell}$ denotes the $\ell$-component of the Weyl vector $\rho(W,f)$.
}
\end{enumerate}

\end{Theorem}
By specializing Theorem \ref{hofthm} in our case, we obtain the main result of this note.
\begin{Theorem}
\label{newbor}
Let $L=\mathbb{Z}[i]\oplus\mathbb{Z}[i]\oplus\frac{1}{2}\mathbb{Z}[i]$ with respect to the standard basis over $\mathbb{Z}[i]$ with hermitian form defined in \eqref{herform}. We set $\ell=(1,0,0)$ and $\ell'=(0,0,1)$. Let $\vec{F}_{m}$ be the vector valued modular form arising from $F_m=\theta_{2}\theta_{1}^{-1}P_{1,m-1}(\varphi_{\infty})$. Then there is a meromorphic function $\Psi(\tau,\sigma; F_{m})= \Psi(\tau,\sigma; \vec F_{m})$ on $\mathcal{H}$ with the following properties:
\begin{enumerate}

\item{$\Psi(\tau,\sigma;\vec F_{m})$ is an automorphic form of weight
$$
32\sum_{d|m}\chi_{-4}(n/d)d^{2}+2\sum_{d|m}\chi_{-4}(d)d^{2}
$$
for $\Gamma_{L}$, with some multiplier system $\chi$ of finite order.}

\item{The zeros and poles of $\Psi(\tau,\sigma;\vec F_{m})$ lie on Heegner divisors. The divisor of $\Psi(\tau,\sigma;\vec F_{m})$ on $X_{\Gamma_L}=\Gamma_{L}\backslash\mathcal{H}$ is given by
$$
div(\Psi(\tau,\sigma;\vec F_{m}))={\bf Z}(m,0) =\Gamma_{L}\backslash {\bf H}(m,0),
$$
where
$$
{\bf H}(m,0)=\sum_{\substack{(r_1,s_1,r_2,s_2,r_3,s_3)\in\mathbb{Z}^{6}\\ r_1r_3+s_1s_3+r_2^2 +s_2^2 =m}}\left\{(\tau,\sigma)\in\mathcal{H}\left|\substack{r_1+2r_2\Re{\sigma}+2s_2\Im{\sigma}+s_3\Re{\tau}-r_3\Im{\tau}=0,\\\\s_1+2r_2\Im{\sigma}-2s_2\Re{\sigma}+s_3\Im{\tau}+r_3\Re{\tau}=0}\right.\right\}.
$$
}

\item{For the Weyl chamber $W_{m}$ described in \eqref{wm},
$\Psi(\tau,\sigma;\vec F_{m})$ has an infinite product expansion near the cusp $\mathbb{Q}e_{3}$ (precisely, when $(\tau,\sigma)\in W_{m,U}$ with $\Im{\tau}$ sufficiently large):
\begin{equation}
\label{AAA}
\Psi(\tau,\sigma;F_{m})=A_1(\tau, \sigma) A_2(\sigma) A_3(\sigma) A_4(\sigma) A_5(\tau, \sigma),
\end{equation}
where

\begin{enumerate}[(i)]

\item{$$A_1(\tau, \sigma)=e(i\rho_{e_{3}}-\rho_{e_{4}}\tau+\bar{\rho}\sigma)$$
where $\rho_{e_{3}}$, $\rho_{e_{4}}$ and $\rho$ are as defined in Subsection \ref{wv},
}

\item{

\begin{align*}
A_2(\sigma)
&=\begin{cases}\displaystyle{
\left[1-e\left(-i\sigma\sqrt{m}\right)\right]}&\text{if $m$ is a square,}\\
\qquad\qquad\qquad1&\text{otherwise,}
\end{cases}
\end{align*}
}

\item{
$$
A_3(\sigma) = \prod_{\substack{(k_{3},k_{4})\in\mathbb{Z}_{>0}^{2}\\ k_{3}^{2}+k_{4}^{2}=m}}
    \left[1-e\left(\sigma\left(k_{3}+ik_{4}\right)\right)\right]\left[1-e\left(\sigma\left(k_{3}-ik_{4}\right)\right)\right],
$$

}

\item{
\begin{align*}
A_4(\sigma) &=\prod_{\substack{n_3, n_4 \in \mathbb Z\\ n_3^2+n_4^2= m}} \prod_{ n_{2}\in \mathbb Z_{>0}}
\left[1-e(i n_{2})e\left(\sigma\left({n_{3}}-i{n_{4}}\right)\right)\right]
\\
  &\quad\times \prod_{ n_{2}\in \mathbb Z_{>0}}(1-e(i n_2))^{c(0,0)}
\end{align*}
with
$$
c(0, 0)=c(0, \phi_0)=\sum_{d|m}\left(64\chi_{-4}(m/d)+4\chi_{-4}(d)\right)d^{2},
$$
}

\item{
$$
A_5(\tau, \sigma)=\prod_{\substack{(n_{1},n_{2},n_{3},n_{4})\in\mathbb{Z}^{4}\\ n_{1}>0}}
\left[1-e\left(n_{1}\tau+\sigma\left(\frac{n_{3}}{2}-i\frac{n_{4}}{2}\right)+in_{2}\right)\right]^{c(n_1n_2- \frac{1}4 (n_3^2+n_4^2),\phi_{\vec n})}
$$
with  $\vec{n}=n_{2}e_{3}-n_1 e_4+\frac{1}{2}(n_{3}+in_{4})$.

}

\end{enumerate}
}

\item{If the cusp corresponding to $\ell$ is neither a pole nor a zero of $\Psi(\tau,\sigma;\vec F_{m})$, then we have
\begin{align*}
\lim_{\tau\to i\infty}\Psi(\tau,\sigma;\vec F_{m})&=e(i\rho_{e_{3}})\prod_{k=1}^{\infty}\left(1-e(ki)\right)^{c(0,\phi_0)}
\end{align*}
where
$$
\rho_{e_{3}}=-\frac{1}{6}\sum_{d|m}\left(16\chi_{-4}(m/d)+\chi_{-4}(d)\right)d^{2}-\frac{1}{24}\sigma_{\chi_{4}}(m)
$$
as defined in subsection \ref{wv},
and
$$
c(0,\phi_0)=\sum_{d|m}\left(64\chi_{-4}(m/d)+4\chi_{-4}(d)\right)d^{2}
$$
as in  Theorem $\ref{vecF}$.
}
\end{enumerate}
\end{Theorem}

\begin{proof}
Assertions (1) and (2) follows directly from Theorem \ref{hofthm} (1) and (2), respectively.

Then by Theorem \ref{hofthm} (3) together with Lemma \ref{wm}, we have that $\Psi(\tau,\sigma;\vec{F}_{m})$ has the following infinite product expansion near the cusp $\mathbb{Q}e_{3}$

\begin{align*}
&\psi(\tau,\sigma;\vec F_{m})\\
&=e(i\rho_{e_{3}}-\rho_{e_{4}}\tau+\bar{\rho}\sigma)\\
&\times\prod_{\substack{(\lambda_{1},\lambda_{2},\lambda_{3},\lambda_{4})\in\mathbb{Z}^{4}\\ \lambda_{2}>0,\\{or\,\, \lambda_{2}=0\,\, and\,\, \lambda_{1}>0,}\\{or \,\,\lambda_{2}=\lambda_{1}=0 \,\,and\,\, \lambda_{3}>0,}\\ {or\,\,\lambda_{2}=\lambda_{1}=\lambda_{3}=0\,\, and \,\, \lambda_{4}>0.}}}\left[1-e\left(\lambda_{2}\tau+\sigma\left(\frac{\lambda_{3}}{2}-i\frac{\lambda_{4}}{2}\right)+i\lambda_{1}\right)\right]^{c\left(\lambda_{1}\lambda_{2}-\frac{1}{4}(\lambda_{3}^{2}+\lambda_{4}^{2}),\,\phi_{\lambda}\right)}
\end{align*}
where $\lambda=\lambda_{1}e_{3}-\lambda_{2}e_{4}+\frac{1}{2}(\lambda_{3}+i\lambda_{4})$, and $\rho_{e_{3}}$, $\rho_{e_{4}}$ and $\rho$ are as defined in Subsection \ref{wv}. We first set $A_1(\tau, \sigma)=e(i\rho_{e_{3}}-\rho_{e_{4}}\tau+\bar{\rho}\sigma)$. Then by decomposing the infinite product according to the four cases in its product index set, we can easily rewrite it as \eqref{AAA}.

Finally, for Assertion (4), we first note that in our case, $K'=\mathbb{Z}i\oplus\mathbb{Z}[i]\oplus\frac{1}{2}\mathbb{Z}i$ and $\delta=2i$, then $\lambda\in K'$ and $\lambda=\frac{1}{2}\kappa\delta\ell=\kappa i\ell$ with $\kappa\in \mathbb{Q}_{>0}$ imply that $\kappa\in\mathbb{Z}_{>0}$ and $c(0,\lambda)=c(0,\phi_{0})$. Together with the Weyl vector attached to $W_{m}$ and $\vec{F}_{m}$ shown in Subsection \ref{wv}, Theorem \ref{hofthm} (5) proves Assertion (4).

\end{proof}


\end{document}